\documentclass{article}
\usepackage{amsmath}
\usepackage{amssymb}
\usepackage{amsthm}
\usepackage[english]{babel}
\usepackage{wrapfig}
\usepackage{graphicx}
\usepackage{hyperref}

\def \tp {{\rm tp}}
\def \qftp {{\rm qftp}}

\def \ind {\mathop{\raisebox{-0.5ex}{$\mathop{\smile \hskip -0.9em ^| \ }$}}\limits}

\def \non-ort {{\not\!\!\bot}}

\def \Th {{\rm Th}}

\def \Q {{\mathbb Q}}

\def \M {{\mathbb{M}}}

\def \s#1 {{\rm S}#1 }

\def \max{{\rm max}}
\def \sup{{\rm sup}}
\def \size{{\rm size}}
\def \calC{{\cal C}}
\def \P{{\cal P}}
\def \IP{{\rm IP}}
\def \NIP{{\rm NIP}}
\def \VC{{\rm VC}}
\def \op{{\rm op}}
\def \opg{{\rm opg}}

\def\proclaim #1. #2\par{\medbreak
  \noindent{\bf#1 \enspace}{\sl#2}\par
  \ifdim\lastskip<\medskipamount \removelastskip\penalty55\medskip\fi}

\newtheorem{theorem}{Theorem}[section]
\newtheorem{definition}[theorem]{Definition}
\newtheorem{proposition}[theorem]{Proposition}
\newtheorem{lemma}[theorem]{Lemma}
\newtheorem{corollary}[theorem]{Corollary}

\newtheorem{remark}[theorem]{Remark}
\newtheorem{example}[theorem]{Example}
\newtheorem{fact}[theorem]{Fact}

\newtheorem{claim}{Claim}[theorem]
\newtheorem*{claim*}{Claim}

\title{On $n$-dependence}
\author{Artem Chernikov \thanks{Partially supported by the Fondation Sciences Math\'{e}matiques de Paris  (FSMP), by a public grant overseen by the French National Research Agency (ANR) as part of the ``Investissements d'Avenir'' program (reference: ANR-10-LABX-0098) and by ValCoMo (ANR-13-BS01-0006).} 
\and Daniel Palacin \thanks{Partially supported by the project MODIG (ANR-09-BLAN-0047) of the French National Research Agency, the project SFB 878, and the project MTM 2011-26840.} 
\and Kota Takeuchi \thanks{Partially supported by JSPS KAKENHI Grant Number 26800077}}

\date{}

\begin{document}

\maketitle
\begin{abstract}
In this note we develop and clarify some of the basic combinatorial properties of the new notion of $n$-dependence (for $1\leq n < \omega$) recently introduced by Shelah \cite{Shelah:vn}. In the same way as dependence of a theory means its inability to encode a bipartite random graph with a definable edge relation, $n$-dependence corresponds to the inability to encode a random $(n+1)$-partite $(n+1)$-hypergraph with a definable edge relation. We characterize $n$-dependence by counting $\varphi$-types over finite sets (generalizing Sauer-Shelah lemma and answering a question of Shelah from \cite{Shelah:kx}) and in terms of the collapse of random ordered $(n+1)$-hypergraph indiscernibles down to order-indiscernibles (which implies that the failure of $n$-dependence is always witnessed by a formula in a single free variable).
\end{abstract}
\section{Introduction}

Shelah had introduced the notion of a \emph{dependent theory} (also called NIP) in his work on the classification program for first-order theories \cite{ShelahClassification}. Since then dependent theories had attracted a lot of attention due to the purely model theoretic work on generalizations of stability and o-minimality (e.g. \cite{HruPill, Sh950, ExtDef2}), the analysis of some important algebraic examples (e.g. \cite{HHM}) and  connections to combinatorics (e.g. \cite{VCdensity_50people}).

More recently, in \cite{Shelah:kx, Shelah:vn} Shelah had introduced a generalization of dependence called $n$-dependence, where $1\leq n < \omega$. The change is that instead of forbidding an encoding of a random bipartite graph with a definable edge relation, one forbids an encoding of a random $(n+1)$-partite $(n+1)$-hypergraph with a definable edge relation (see Definition \ref{def: n-dependence}). Then dependence corresponds to $1$-dependence, and we have an increasing family of classes of theories.

So far, not much is known about $n$-dependent theories. In  \cite{Shelah:vn} Shelah demonstrates some results about connected components for (type)-definable groups in 2-dependent theories (which can be viewed as a form of modularity in certain context, see remarks in \cite[Section 6.5]{hrushovski2013}). In \cite{Hempel} Hempel shows a finitary version of this result giving a certain ``chain condition'' for groups definable in $n$-dependent theories and demonstrating that every $n$-dependent field is Artin-Schreier closed. Some further questions and statements are mentioned in \cite[Section 5(H)]{Shelah:kx}. The aim of this note is essentially to clarify that material and to answer some questions posed there. Here is the outline of the paper. 

In Section \ref{sec: n-dependence} we define $n$-dependence of a formula and give some motivating examples of $n$-dependent theories.

In Section \ref{sec: counting types} we introduce a generalization of VC-dimension capturing $n$-dependence and give a corresponding generalization of Sauer-Shelah lemma using bounds on the so-called Zarankiewicz numbers for hypergraphs from combinatorics. As an application we characterize $n$-dependent theories by counting $\varphi$-types over finite sets and give a counterexample to a more optimistic bound asked by Shelah. A preliminary version of the upper bound result has appeared in \cite{Kota}. The optimality of our result remains open (and is closely connected to the open problem of lower bounds for Zarankiewicz numbers).

In Section \ref{sec: Gen Indisc} we discuss existence of various generalized indiscernibles useful for the study of $n$-dependence and connections to some results from structural Ramsey theory. In Section \ref{sec: IPn and gen indisc} we apply these observations to show that a theory is $n$-dependent if and only if every ordered random $(n+1)$-hypergraph indiscernible is actually just order-indiscernible. The case $n=1$ is due to Scow \cite{Scow:2011uq}. 

Another application of hypergraph indiscernibles is given in Section \ref{sec: reduction to 1 variable} where we demonstrate that a theory is $n$-dependent if and only if every formula in a \emph{single} free variable is $n$-dependent. This is a result due to Shelah \cite[Claim 2.6]{Shelah:vn}, however the authors found the proof suggested there to be lacking in some details and we use this opportunity to provide a detailed account of Shelah's theorem.

Finally, in the Appendix we verify a claim from Section \ref{sec: Gen Indisc} that the class of ordered partite hypergraphs forms a Ramsey class. This might be folklore, but we feel that a readable account could be beneficial.

\section{$n$-dependence} \label{sec: n-dependence}

The following property was introduced in \cite[Section 5(H)]{Shelah:kx} and \cite[Definition 2.4]{Shelah:vn}.

\begin{definition}\rm \label{def: n-dependence}
A formula $\varphi\left(x;y_{0},\ldots,y_{n-1}\right)$ has the \emph{$n$-independence property}, or $\IP_{n}$ (with respect to a theory $T$), if in some model there is a sequence $\left(a_{0,i},\ldots,a_{n-1,i}\right)_{i\in\omega}$
such that for every $s\subseteq\omega^{n}$ there is $b_{s}$ such
that 
\[
\models\phi\left(b_{s};a_{0,i_{0}},\ldots,a_{n-1,i_{n-1}}\right)\Leftrightarrow\left(i_{0},\ldots,i_{n-1}\right)\in s\mbox{.}
\]
Here $x,y_0, \ldots, y_{n-1}$ are possibly tuples of variables.
Otherwise we say that $\varphi\left(x,y_{0},\ldots,y_{n-1}\right)$ is $n$-dependent, or 
$\NIP_{n}$.
A theory is \emph{$n$-dependent}, or $\NIP_n$, if it implies that every formula is
$n$-dependent.
\end{definition}

We give some motivating examples and remarks.
\begin{example}\label{ex: basic n-dependent} \rm
\begin{enumerate}
\item If $T$ is $n$-dependent then it is $\left(n+1\right)$-dependent. Of course, $T$ is dependent if and only if it is $1$-dependent.
\item The theory of a random $n$-hypergraph is $\left(n+1\right)$-dependent,
but not $n$-dependent. Here $(n+1)$-dependence is immediate by quantifier elimination and Proposition \ref{prop: criterion for n-dependence}, and $n$-independence is witnessed by the edge relation. The same holds for random $n$-partite $n$-hypergraphs and for random $K_{m}$-free $n$-hypergraph.

\item Similarly, it follows by the type-counting criterion from Proposition \ref{prop: criterion for n-dependence} that in fact any theory with elimination of quantifiers in which any atomic formula has at most $n$ variables is $n$-dependent. In particular, any theory eliminating quantifiers in a finite relational language is $n$-dependent, where $n$ is the maximum of the arities of the relations in the language.

\item A theory $T$ is called quasifinite if there is a function $\nu : \omega \to \omega$ such that every finite subset $T_0$ of $T$ has a finite model in which the number of $k$-types is bounded by $\nu(k)$. In particular, every quasifinite theory is pseudofinite and $\aleph_0$-categorical. Quasifinite theories are studied in depth in \cite{CherlinHrushovski}, and in \cite[Section 6.5]{hrushovski2013} it is pointed out that every quasi-finite theory is $2$-dependent: it is demonstrated in \cite{CherlinHrushovski} using the classification of finite simple groups that in a quasifinite theory, $\pi_{\Delta}(m)$ grows at most as $2^m$ (see Definition \ref{def: pi_phi} and Proposition \ref{prop: criterion for n-dependence}). An example of a quasifinite theory is the theory of a generic bilinear form on an infinite-dimensional vector space over a finite field (a direct proof that this theory is $2$-dependent is given in \cite{Hempel}).

\item On the other hand, any theory of an infinite boolean algebra is $n$-independent, for all $n$ (see \cite[Example 2.10]{Shelah:vn}).

\item By a result of Beyarslan \cite{beyarslan2010random}, any pseudo-finite field interprets random $n$-hypergraph, for all $n$ --- so it is not $n$-dependent for any $n$. More generally, \cite{Hempel} shows that any PAC field which is not separably closed is $n$-independent, for all $n$.  In view of this (and the well-known conjecture that all supersimple fields are PAC), one could ask if in fact every (super)simple $n$-dependent field is separably closed.
\end{enumerate}
\end{example}

\section{Counting $\varphi$-types and a generalization of Sauer-Shelah lemma} \label{sec: counting types}

\subsection{Sauer-Shelah Lemma and Generalized VC-dimension}
The maximum number of $\varphi(x,y)$-types over finite sets coincides with the value of the shatter function in the theory of VC-dimension in combinatorics (see e.g. \cite{VCdensity_50people} for a detailed account of this correspondence).
We generalize the notion of VC-dimension and investigate the upper bound of the generalized shatter function.
In this subsection, we discuss purely combinatorial topics. 
The connection with counting $\varphi$-types and $n$-dependence will be discussed in the next subsection
(see Lemma \ref{coincide1} and Lemma \ref{coincide2}).

First we recall classical VC-dimension and Sauer-Shelah lemma.
Let $X$ be a set and $\calC\subseteq \P(X)$ a class of subsets of $X$. Given a subset $A\subseteq X$ we write $\calC \cap A$ to denote the set $\{C\cap A: C\in\calC\}$.
\begin{definition}[Vapnik and Chervonenkis]\rm
A subset $A\subseteq X$ is said to be \emph{shattered} by $\calC$ if $\calC \cap A =\P(A)$. The VC-dimension of $\calC$ is defined as 
$$\VC(\calC)=\sup\{|A|:A\subseteq X \mbox{ is shattered by } \calC\},$$ and the shatter function of $\calC$ is defined as $$\pi_\calC(m):=\max\{|\calC \cap A|: A\subseteq X, |A|=m\}.$$ Observe that $0\leq \pi_\calC(m)\leq 2^m$, and $\pi_\calC(m)=2^m$ if and only if $m \leq \VC(\calC)$.
\end{definition}


\begin{fact}\rm
\begin{enumerate}
\item
(Sauer-Shelah lemma) Assume that $\VC(\calC) \leq d$. Then $\pi_\calC(m)\leq \sum_{i\leq d}\binom{m}{i}$
 for all $m \geq d$.
In particular, $\pi_\calC(m)\leq \left(\frac{em}{d} \right)^d = O(m^d)$ for all $m$.
\item
There is a class $\calC\subseteq \P(X)$ with $\VC(\calC)=d$ such that $\pi_\calC(m)=\sum_{i\leq d}\binom{m}{i}$ for $m\geq d$ (e.g. the class of all subsets of $X$ of size $\leq d$). Hence, the bound given by Sauer-Shelah lemma is tight.
\end{enumerate}
\end{fact}

Throughout this subsection, we fix (infinite) sets $X_0,\ldots ,X_{n-1}$ and $X={\prod}_{i<n}X_i$. 
For a class $\calC\subseteq \P(X)$ we define a notion of $\VC_n$-dimension of $\calC$.

\begin{definition}\rm
A subset $A\subseteq X$ is said to be a box of $\size(A)=m$ if
$A={\prod}_{i<n} A_i$ for some $A_i\subseteq X_i$ ($i<n$) with each $|A_i|=m$.
The $\VC_{n}$-dimension of $\calC$ is defined as 
$$\VC_n(\calC)= \sup\{ \size(A):A\subseteq X \mbox{ is a box shattered by } \calC\},$$ and the corresponding shatter function by $$\pi_{\calC,n}(m):=\max\{|\calC \cap A|: \mbox{$A\subseteq X$ is a box of size $m$}\}.$$
\end{definition}

\begin{remark}\rm
\begin{enumerate}
\item
$0\leq \pi_{\calC,n}(m)\leq 2^{m^n}$.
\item
$\pi_{\calC,n}(m)=2^{m^n}$ if and only if $m \leq \VC_n(\calC)$.
\end{enumerate}
\end{remark}

We generalize Sauer-Shelah lemma below. First we introduce some notation from extremal graph theory.
\begin{definition}\rm
Let $G^{(n)}(m_0,\ldots ,m_{n-1})$ denote an $n$-partite $n$-uniform hypergraph such that the $i$-th part has $m_i$ vertices. If $m_0=\ldots =m_{n-1}=m$, we simply write $G^{(n)}(m)$. Moreover, let $K^{(n)}(m)$ be the complete $n$-partite $n$-uniform hypergraph $G^{(n)}(m)$. (For example, $K^{(2)}(3)$ is the bipartite complete graph $K_{3,3}$.) Then:
\begin{itemize}
\item
The value ${\rm ex}_n(m, K^{(n)}(d))$ is the minimum natural number $k$ satisfying the following:
for every (not partite) $n$-uniform hypergraph $G$ with $m$-vertices, if $G$ has $\geq k$ edges then $G$ contains $K^{(n)}(d)$ as a subgraph.
\item
The Zarankiewicz number $z_n(m,d)$ is the minimum natural number $z$ satisfying the following:
every $G^{(n)}(m)$ having $\geq z$ edges contains $K^{(n)}(d)$ as a subgraph.
\end{itemize}
\end{definition}

\begin{fact}\rm\cite{Erdos}\label{Erdos}
For given $n$ and $d$, let $\varepsilon=\dfrac{1}{d^{n-1}}$. Then there is $k\in\omega$ such that for every $m>k$ we have:
\begin{enumerate}
\item
${\rm ex}_n(m, K^{(n)}(d))\leq m^{n-\varepsilon}$,
\item
in particular $z_n(m,d)\leq (nm)^{n-\varepsilon}$.
\end{enumerate}
\end{fact}

It is known that the bound given above is tight for $n=2$ and $d = 2,3$ (see e.g. \cite{pach2011combinatorial}), but the question about lower bounds is widely open even for graphs in general (the best lower bound for $n=2$ and $d \geq 5$ is $\Omega(n^{2-2/d} \log(d)^{1/(d^2-1)})$ \cite{bohman2010early}).
For our purposes we will only need the following:
\begin{fact}\label{lower bound}\rm \cite[Chapter 5.2, Corollary 2.7]{Bollobas}
There is $k\in \omega$ such that $z_2(m,2)>m^{3/2}(1-\dfrac 1{m^{1/6}})$ for every $m>k$. 

In particular, we can find $c>0$ such that $z_2(m,2)\geq cm^{3/2}$ for every $m\in\omega$. 
\end{fact}

In order to generalize Sauer-Shelah lemma we need the so-called ``shifting technique" lemma from combinatorics (see e.g. $\cite{Shifting})$.
\begin{fact}[Shifting technique]\label{shifting}\rm
Let $A$ be any finite set and $\calC\subseteq \P(A)$.
Then there is $\calC'\subseteq \P(A)$ such that:
\begin{enumerate}
\item
$|\calC|=|\calC'|$,
\item
if $\calC'$ shatters $B\subseteq A$ then so does $\calC$,
\item
if $B\subseteq C\in \calC'$ then $B\in\calC'$.
\end{enumerate}
\end{fact}


\begin{proposition}\label{generalized Sauer}\rm
Let $\calC$ be a class of subsets of $X$. 
\begin{enumerate}
\item
Assume that $\VC_n(\calC) \leq d$. 
Then $\pi_{\calC,n}(m)\leq \sum_{i<z}\binom{m^n}{i}$ for $m\geq d$, where $z=z_n(m,d+1)$.

\item In particular, for  $m\gg n,d$, we have
 $\pi_{\calC,n}(m)\leq 2^{cm^{n-\varepsilon}\log_2 m} \leq 2^{m^{n - \varepsilon'}}$, where $c=n^{n+1-\varepsilon}$, $\varepsilon = \dfrac{1}{(d+1)^{n-1}}$ and $\varepsilon' = \varepsilon'(n,d) >0 $ is small enough.
\item
There is a class $\calC\subseteq \P(X)$ with $\VC_n(\calC)=d$ such that $\pi_{\calC,n}(m)\geq 2^{z-1}$ where $z=z_n(m,d+1)$. 
\end{enumerate}
\end{proposition}

Note that the first item in the proposition gives Sauer-Shelah lemma where $n=1$,
since $z_1(m,d+1)=d+1$.
In addition, by Fact \ref{lower bound}, we can find a class $\calC$ and $c>0$ such that $\pi_{\calC,2}(m)\geq 2^{cm^{3/2}}$ for every $m$.
Unfortunately, the inequality $\pi_{\calC,n}(m)\leq \sum_{i<z}\binom{m^n}{i}$ may not be tight.
\begin{proof}[Proof of Proposition \ref{generalized Sauer}]
(1):
 We show that $\pi_{\calC,n}(m)\leq \sum_{i<z}\binom{m^n}{i}$.
Let $A\subseteq X$ be a box of size $m$.
It is enough to show that $|\calC \cap A|\leq\sum_{i<z}\binom{m^n}{i}$.
Let $\calC'$ be given by Lemma \ref{shifting} applied to $\calC \cap A$. By the third condition in the lemma, 
every $B\in\calC'$ is shattered by $\calC'$, and moreover $\calC$ shatters all members in $\calC'$ by the second condition. Hence 
$\calC'$ contains no box of size $d+1$ since $\VC_n(\calC) \leq d$.

\begin{claim*}\rm
For $B\in \calC'$, $|B|<z_n(m,d+1)$.
\end{claim*}
\begin{proof}[Proof of the claim] Suppose that $|B|\geq z_n(m,d+1)$.
Consider the $n$-partite $n$-uniform hypergraph $G=(A_0\sqcup\ldots \sqcup A_{n-1}; B)$ (recall that $A={\prod}_{i<n}A_i$ is a box of size $m$).
Then $G$ has a subgraph $G'\cong K^{(n)}(d+1)$ by Fact \ref{Erdos}.
Notice that the set of edges $B'$ of $G'$ (i.e. $B'=E(G')$)  is a subset of $B$, hence $B'$ is shattered by $\calC$.
However, $B'$ is a box of size $d+1$. This contradicts the fact that $\VC_n(\calC) \leq d$.
\end{proof}

Therefore $\calC'\subseteq \{B\subseteq A: |B|<z_n(m;d+1)\}$ and so, 
$$
|\calC \cap A|=|\calC'|\leq|\{B\subseteq A: |B|<z_n(m;d+1)\}|\leq\sum_{i<z}\binom{m^n}{i}.
$$

(2): A straightforward calculation using (1) and Fact \ref{Erdos}.

(3):
Without loss of generality, we may assume $X_i=\omega$ for all $i<n$, since the shatter function $\pi_{\calC,n}$ is determined locally, i.e. if $X\subseteq X'$ and $\calC' = \calC$ is a family of subsets of $X'$ then $\pi_{\calC', n}(m)=\pi_{\calC,n}(m)$.
For each $m\in\omega$ with $m\geq d$, let $(A^m_0\cup\ldots \cup A^m_{n-1}; E^m)$ be an $n$-partite $n$-uniform hypergraph having $(z_n(m,d+1)-1)$-edges with no subgraph isomorphic to $K^{(n)}(d+1)$.
We may assume that $X_i$ is the disjoint union of $A_i^m$ $(m\geq d)$.
Let $\calC = \bigcup_m \P(E^m)$.
Clearly, we have $\pi_{\calC,n}(m)\geq 2^{z_n(m,d+1)-1}$ and $\VC_n(\calC)\geq d$.
On the other hand, since every $C\in \calC$ is in some 
$\P(E^m)$,
every box $B$ shattered by $\calC$ must be a subset of some ${\prod}_iA_i^m$.
This means $\VC_n(\calC)\leq d$. 
\end{proof}

\subsection{$VC_n$-dimension and $n$-dependence}
In this subsection, we translate our situation with an $n$-dependent formula into the theory of $\VC_n$-dimension.

\begin{definition}\rm \label{def: pi_phi}
Let $\Delta(x,y_0,\ldots ,y_{n-1})$ be a set of formulas with $|y_k|=l_k$,
 and  for each $k<n$, let 
$A_k$ be a (small) set of {\it tuples} of length $l_k$ in a monster model.
A (complete) $\Delta$-type $p(x)$ over $(A_0,\ldots ,A_{n-1})$ is a (maximal) consistent subset of
$\{\varphi(x,a_0,\ldots ,a_{n-1})^{\rm{if} (i=0)}: \varphi\in\Delta, a_k\in A_k, i<2\}$, and $S_\Delta(A_0,\ldots ,A_{n-1})$ is the set of complete $\Delta$-types over $(A_0,\ldots ,A_{n-1})$. 
For a natural number $m\in \omega$, we put
$$\pi_\Delta(m):= \sup\{|S_\Delta(B_0,\ldots ,B_{n-1})|: |B_0|=\ldots =|B_{n-1}|=m\}.$$
If $\Delta $ consists only of a single formula $\phi$, then we simply write $S_\phi, \pi_\phi$, etc.

\end{definition}

\begin{remark}\label{remark on number of types}\rm
Let $\varphi(x,y_0,\ldots ,y_{n-1})$ and $\psi(x,y_0,\ldots ,y_{n-1})$ be formulas.
\begin{enumerate}
\item
$0\leq\pi_\varphi(m)\leq 2^{m^n}$ for every $m\in\omega$.
\item
\begin{enumerate}
\item
$\pi_{\neg\varphi}(m)=\pi_\varphi(m)$.
\item
$\pi_{\varphi\wedge\psi}(m)\leq\pi_{\{\varphi,\psi\}}(m)\leq \pi_{\varphi}(m)\cdot\pi_\psi(m)$.
\end{enumerate}
\item
The following are equivalent:
\begin{enumerate}
\item
$\varphi$ is $n$-dependent.
\item
$\pi_\varphi(m)<2^{m^n}$ for some $m\in\omega$.
\item
There is $d\in\omega$ such that $\pi_\varphi(m)=2^{m^n}$ for $m \leq d$ and $\pi_\varphi(m)<2^{m^n}$ for $m > d$. 
\end{enumerate}
\end{enumerate}
\end{remark}
We call 
the 
number 
$d$ in condition (3c) the (dual $\VC_n$-)dimension of $\varphi$. The dimension of $\varphi$ will be denoted by $\dim(\varphi)$. 
In \cite[Section 5(H), Question 5.67(1)]{Shelah:kx}, Shelah asks whether the following condition $(*)$ is equivalent to $n$-dependence of $\varphi(x,y_0,\ldots ,y_{n-1})$:
\begin{itemize}
\item[$(*)$] There is $k\in \omega$ such that
$\pi_\varphi(m) \leq 2^{cm^{n-1}}$ for all $m>k$,
\end{itemize}
where $c=|x|$.
Clearly $(*)$ implies $n$-dependence, however $(*)$ is too strong to be equivalent to it. In fact, as stated it is trivially false for $n=1$. But even if we fix the $n=1$ case by replacing $2^{cm^{n-1}}$ with $2^{cm^{n-1}\log_2 m}$, it is still too strong for larger $n$, as the following theorem demonstrates.
\begin{theorem}\label{counting types}\rm
\begin{enumerate}
\item
(Weak form of $(*)$)
If $\varphi(x,y_0,\ldots ,y_{n-1})$ is $n$-dependent with $\dim(\varphi) \leq d$ then
$\pi_\varphi(m)\leq \sum_{i<z}\binom{m^n}{i}$ for $d\leq m$, where $z=z_n(m,d+1)$. In particular, for $m \gg n,d$ we have $\pi_\varphi(m) \leq 2^{cm^{n-\varepsilon}\log_2 m} \leq 2^{m^{n-\varepsilon'}}$
where $c=n^{n+1-\varepsilon}$, $\varepsilon=\dfrac{1}{(d+1)^{n-1}}$ and $\varepsilon' = \varepsilon'(n,d) > 0$ is small enough.
\item
(Counterexample for Shelah's question):
There are $c>0$, a theory $T$ and a formula $\varphi(x,y_0,y_1)$ such that
$\varphi$ is $2$-dependent and $\pi_\varphi(m)\geq 2^{cm^{3/2}}$ for every $m$.
\end{enumerate}
\end{theorem}
We first need to introduce some notation and generalize some standard observations from $n=1$ to arbitrary $n$.
Let $\varphi(x,y_0,\ldots ,y_{n-1})$ be a formula and fix a model $M$ of $T$. 
For simplicity of notation we will assume all variables to be of length $1$, for example, $|x|=|y_0|=|a_0|=1$.
The class ${\cal C}_\varphi$ is defined as  $$\calC_{\varphi}=\{\varphi(b,M^n): b\in M\}\subseteq 2^{M^n},$$ where $\varphi(b,M^n)=\{(a_0,\ldots ,a_{n-1})\in M^n: M\models\varphi(b,a_0,\ldots ,a_{n-1})\}$. 

\begin{lemma}\label{coincide1}\rm
For every finite $A\subseteq M^n$,
we have $|S_\varphi(A)|=|{\cal C}_{\varphi} \cap A|$, where $S_\varphi(A)$ is the set of all complete $\phi$-types. (A complete $\phi$-type over $A$ is a maximal consistent set of formulas of the form $\varphi(x,a_0,\ldots ,a_{n-1})$ or $\neg \varphi(x,a_0,\ldots ,a_{n-1})$ with $(a_0,\ldots ,a_{n-1})\in A$.)
\end{lemma}
\begin{proof}
Note that since $A$ is finite, $p$ is realized in $M$ for every $p\in S_\varphi$.
Consider 
the map from $S_\varphi$ to $\calC_{\varphi} \cap A$ given by $p(x)\mapsto \varphi(b,M^n)\cap A$ for some $b\models p$.
Notice that $b$ and $b'$ satisfy the same type $p$ if and only if for every $(a_0,\ldots ,a_{n-1})\in A$, $\varphi(b,a_0,\ldots ,a_{n-1})\leftrightarrow\varphi(b',a_0,\ldots ,a_{n-1})$ holds.
Hence the map is well-defined and injective.
Moreover, if $\varphi(b,M^n)\cap A\in {\cal C}_{\varphi} \cap A$ then we can find $\tp_\varphi(b/A)\in S_\varphi(A)$. So the map is a bijection.
\end{proof}
The above lemma shows that there is no difference between counting types and counting the size of the restricted class. Hence, by the definition, we have:
\begin{lemma}\label{coincide2}\rm
\begin{enumerate}
\item
$\pi_\varphi(m)=\pi_{{\cal C}_\varphi,n}(m)$ for every $m\in\omega$.
\item
$\dim(\varphi)=\VC_n({\cal C}_\varphi)$. 
In particular, a formula $\varphi$ is $n$-dependent if and only if the
$\VC_n$-dimension of ${\cal C}_\varphi$ is finite.
\end{enumerate}
\end{lemma}

Note that $\pi_{\varphi}$ and $\dim(\varphi)$ do not depend on the model inside which they are calculated, thus they are indeed properties of a formula. 
Now, we give a proof of our theorem.
\begin{proof}[Proof of Theorem \ref{counting types}]
(1): Immediate from Proposition \ref{generalized Sauer} and Lemma \ref{coincide2}.

(2): By the second item of Proposition \ref{generalized Sauer} and the remark after  
its statement, for countable sets $X_0$ and $X_1$, we can find $c>0$ and $\calC\subseteq \P(X_0\times X_1)$ such that $\VC_2(\calC)=1$ and $\pi_{\calC,2}(m)\geq 2^{cm^{3/2}}$ for all $m$.
We may assume that $X_0=X_1$.
With a set $Y=\{b_C:C\in\calC\}$, we define a structure $M=(Y\cup X_0, R(x,y_0,y_1))$ by the following:
$R(b,a_0,a_1)$ if and only if $(a_0,a_1)\in C\subseteq X_0^2$ and $b=b_C$ for some $C\in\calC$.
Then, in $M$, we have $\calC_R = \calC$, hence $\pi_{R}(m)=\pi_{\calC,2}(m)$.
\end{proof}

\begin{corollary}\rm\label{boolean combinations}
Let $\varphi(x,y_0,\ldots ,y_{n-1})$ and $\psi(x,y_0,\ldots ,y_{n-1})$ be $n$-dependent formulas.
Then $\neg\varphi$, $\varphi\wedge\psi$ and $\varphi\vee\psi$ are $n$-dependent.
\end{corollary}
\begin{proof}
Immediate from Remark \ref{remark on number of types} and Theorem \ref{counting types}, since
$$
\log (\pi_\varphi(m)\pi_\psi(m))=\log(\pi_\varphi(m)) + \log(\pi_\psi(m))=O(m^{n-\varepsilon})
$$
for some $\varepsilon>0$.
\end{proof}

\section{Generalized indiscernibles} \label{sec: Gen Indisc}

\subsection{Ramsey property and hypergraphs} 
In this subsection, we arrange several facts in structural Ramsey theory with hypergraphs. We postpone some of the proofs until the appendix.

Let $L_0$ be a finite relational language, and let $A,B,C$ be $L_0$-structures. We denote by $\binom{B}{A}$ the set of all $A'\subseteq B$ such that $A'\cong_{L_0} A$. If $A$ has no non-trivial automorphisms, then $\binom{B}{A}$ is considered as a set of all embeddings $A\to B$. Using this notation, for $k\in \omega$, we write $C\to(B)^A_k$ to denote the following property: for every map $c:\binom{C}{A}\to k$ (called a coloring) there is $B'\in\binom{C}{B}$ such that $c|\binom{B'}{A}$ is constant.


\begin{definition}\rm
Let $K$ be a set of (the isomorphism types of) $L_0$-structures and let $A,B\in K$.
We say that $K$ has the $(A,B)$-Ramsey property if  for every $k\in \omega$ there is $C\in K$ such that $C\to (B)^A_k$.
In addition, if $K$ has the $(A,B)$-Ramsey property for every $A,B\in K$, then we say that $K$ has Ramsey property (or it is a Ramsey class).
\end{definition}

We introduce three Ramsey classes: ordered $n$-partite sets, ordered $n$-uniform hypergraphs and ordered $n$-partite $n$-uniform hypergraphs.

Let $L_{\op}=\{<,P_0(x),\ldots ,P_{n-1}(x)\}$.
An ordered $n$-partite set is an $L_{\op}$-structure $A$ such that $A$ is the disjoint union of $P_0(A),\ldots ,P_{n-1}(A)$ and that $<$ is a linear ordering on $A$ with $P_0(A)<\ldots <P_{n-1}(A)$.
\begin{fact}[see Appendix A, Proposition \ref{partite ordered sets has Ramsey property}]\label{Ramsey property of ordered partite sets}\rm
Let $K$ be the set of all finite ordered $n$-partite sets and let $K^*=\{A:A\subseteq B\in K\}$ be the hereditary closure of $K$. 
Then $K$ and $K^*$ are Ramsey classes.
The Fra\"{i}ss\'e limit of $K^*$ will be denoted by $O_{n,p}$.
\end{fact}

Let $L_{0}=\left\{ R_{i}\right\} _{i\in I}$
be a finite relational signature, let $n_{i}$ be the arity of $R_{i}$.
A \emph{hypergraph of type $L_{0}$} is a structure $\left(A,\left(R_{i}^{A}\right)_{i\in I}\right)$
such that for all $i\in I$:
\begin{itemize}
\item $R_{i}\left(a_{0},\ldots,a_{n_{i}-1}\right)$ $\Rightarrow$ $a_{0},\ldots,a_{n_{i}-1}$
are distinct,
\item $R_{i}\left(a_{0},\ldots,a_{n_{i}-1}\right)$ $\Rightarrow$ $R\left(a_{\sigma\left(0\right)},\ldots,a_{\sigma\left(n_{i} - 1\right)}\right)$
for any permutation $\sigma\in\mbox{Sym}\left(n_{i}\right)$.
\end{itemize}
Thus essentially $R_{i}^{A}\subseteq\left[A\right]^{n_{i}}$, the
set of subsets of $A$ of size $n_{i}$. Let $\mbox{OH}_{L_{0}}$
be the set of all \emph{(linearly) ordered} $L_{0}$-hypergraphs, it is a Fra\"{i}ss\'e
class and admits a Fra\"{i}ss\'e limit --- the ordered random $L_{0}$-hypergraph,
with the order isomorphic to $\left(\mathbb{Q},<\right)$. 
In particular, an ordered $L_0$-hypergraph is called an ordered $n$-uniform hypergraph if $L_0=\{R(x_0,\ldots ,x_{n-1})\}$, and $G_n$ denotes the countable ordered $n$-uniform random hypergraph.
It is proved
in \cite{MR0437351,MR692827} and independently in \cite{MR503795}
that:
\begin{fact}\label{fac: hypergraphs are Ramsey}\rm
For any finite $L_{0}$, the class of
all ordered $L_{0}$-hypergraphs $\mbox{OH}_{L_{0}}$ is a Ramsey class.
\end{fact}

Fix a language $L_{\opg}=\{R(x_0,\ldots ,x_{n-1}),<, P_0(x),\ldots , P_{n-1}(x)\}$. 
A (linearly) ordered $n$-partite $n$-uniform hypergraph is an $L_{\opg}$-structure $\left(A; <, R, P_0, \ldots , P_{n-1} \right)$ such that:
\begin{enumerate}
\item
$(A;R,P_0,\ldots ,P_{n-1})$ is an $n$-partite $n$-uniform hypergraph, i.e. $A$ is the (pairwise disjoint) union $P_0 \sqcup\ldots \sqcup P_{n-1}$ such that if $(a_0,\ldots ,a_{n-1})\in R$ then $P_i\cap \{a_0\ldots a_{n-1}\}$ is a singleton for every $i<n$,
\item
$<$ is a linear ordering on $A$ with $P_0(A)<\ldots <P_{n-1}(A)$.
\end{enumerate}

\begin{fact}[see Proposition \ref{ramsey property of partite graphs} and Lemma \ref{hereditary closure}]\label{Ramsey property of ordered partite hypergraphs}\rm
Let $K$ be the set of all finite ordered $n$-partite $n$-uniform hypergraphs and let $K^*=\{A:A\subseteq B\in K\}$ be the hereditary closure of $K$.
Then both $K$ and $K^*$ have Ramsey property.
\end{fact}
The Fra\"{i}ss\'e limit of $K^*$ is called an ordered $n$-partite $n$-uniform random hypergraph, denoted by $G_{n,p}$.

\begin{remark}\label{Theory of random graph}\rm
The first order theories of $G_n$ and $G_{n,p}$ can be axiomatized in the following way:
\begin{enumerate}
\item
A structure $(M,<,R)$ is a model of $\Th(G_n)$ if and only if:
\begin{itemize}
\item
$(M,<,R)$ is an ordered $n$-uniform hypergraph,
\item
$(M,<)$ is DLO,
\item
for every finite disjoint sets $A_0,A_1\subset M^{n-1}$ and
$b_0<b_1\in M$, there is $b_0<b<b_1$ such that $R(b, a_{i,1},\ldots ,a_{i,n-1})^{{\rm if}(i=0)}$ for every $(a_{i,1},\ldots ,a_{i,n-1})\in A_i$ and $i<2$.
\end{itemize}
In particular, an ordered random $1$-hypergraph is a dense linear order with a dense co-dense subset. 

\item
A structure $(M,<,R,P_0,\ldots ,P_{n-1})$ is a model of $\Th(G_{n,p})$ if and only if:
\begin{itemize}
\item
$(M,<,R,P_0,\ldots ,P_{n-1})$ is an ordered $n$-partite $n$-uniform hypergraph,
\item
$(P_i(M), <)$ is DLO for each $i<n$,
\item
for every $j<n$, finite disjoint sets 
$A_0,A_1\subset {\prod_{i\neq j}P_i(M)}$
 and
$b_0<b_1\in P_j(M)$, there is $b_0<b<b_1$ such that $R(b,a_{i,1},\ldots ,a_{i,n-1})^{{\rm if}(i=0)}$ for every $(a_{i,1},\ldots ,a_{i,n-1})\in A_i$ and $i<2$.
\end{itemize}
\item
$G_{n,p}|L_{\op}$ is isomorphic to $O_{n,p}$.
\item
$\Th(G_n)$, $\Th(G_{n,p})$, and $\Th(O_{n,p})$ are $\omega$-categorical and admit quantifier elimination.  
\end{enumerate}
\end{remark}

\subsection{Generalized indiscernibles}
The notion of generalized indiscernibles, which was introduced in \cite[Section 2]{Scow:2012fk}, and was used implicitly by Shelah already in \cite{ShelahClassification}, is a good tool to study $n$-dependence.
\begin{definition}\rm
Let $T$ be a theory in the language $L$, and let $\M$ be a monster
model of $T$.
\begin{enumerate}
\item Let $I$ be a structure in the language $L_0$.
We say that $\bar{a}=\left(a_{i}\right)_{i\in I}$
with $a_{i}\in\M$ is an $I$-indiscernible if for
all $n\in\omega$ and all $i_{0},\ldots,i_{n}$ and $j_{0},\ldots,j_{n}$
from $I$ we have:
$$
\qftp_{L_0}\left(i_{0}, \ldots, i_{n}\right)=\qftp_{L_0}\left(j_{0},\ldots,j_{n}\right)
\Rightarrow $$
$$\tp_{L}\left(a_{i_{0}}, \ldots, a_{i_{n}}\right)=\tp_{L}\left(a_{j_{0}}, \ldots, a_{j_{n}}\right)
.$$
An $I$-indiscernible $(a_i)_{i\in I}$ is also called an $L_0$-indiscernible to clarify the structure on $I$. 
(So, for $L_1\subseteq L_0$, $(a_i)_{i\in I}$ is said to be $L_1$-indiscernible if it is $(I|L_1)$-indiscernible.)
\item For $L_0$-structures $I$ and $J$, we say that $\left(b_{i}\right)_{i\in J}$
is \emph{based on} $\left(a_{i}\right)_{i\in I}$ if for any finite
set $\Delta$ of $L$-formulas,  and for any finite tuple $\left(j_{0},\ldots,j_{n}\right)$
from $J$ there is a tuple $\left(i_{0},\ldots,i_{n}\right)$ from
$I$ such that:

\begin{itemize}
\item $\qftp_{L_0}\left(j_{0},\ldots,j_{n}\right)=\qftp_{L_0}\left(i_{0},\ldots,i_{n}\right)$ and
\item $\tp_{\Delta}\left(b_{j_{0}},\ldots,b_{j_{n}}\right)=\tp_{\Delta}\left(a_{i_{0}},\ldots,a_{i_{n}}\right)$.
\end{itemize}
\item \cite{Scow:2012fk} Let $I$ be a structure in the language $L_0$. We say that $I$ has the \emph{modeling property} if given any $\bar{a}=\left(a_{i}\right)_{i\in I}$
there exists an $L_0$-indiscernible $\bar{b}=\left(b_{i}\right)_{i\in I}$
based on $\bar{a}$.
\end{enumerate}
\end{definition}

For a class $K$ of $L_0$-structures, we say an $L_0$-structure $G$ is $K$-universal if for every $A\in K$ there is $A'\subset G$ such that $A\cong A'$.
\begin{fact}\label{modeling property}\rm \cite{Scow:2012fk} 
Let $K$ be a class of finite $L_0$-structures and let $G$ be a countable $K$-universal $L_0$-structure such that $A\in K$ for every finite $A\subset G$.
Then $K$ is a Ramsey class if and only if $G$ has the modeling property. 
\end{fact}
\begin{proof} We prove left to right for the sake of exposition.
Take any finite subsets $A\subset B\subset G$ and a formula $\varphi((x_g)_{g\in A})$. 
Since $A,B\in K$ and $K$ is a Ramsey class, there is $C\in K$ such that $C\to(B)^A_2$.
By the assumption, we may assume $C\subset G$.
Hence we can find $(a_g)_{g\in B'}\subset (a_g)_{g\in G}$ with $B'\in \binom{C}{B}$ such that for any $A', A''\in \binom{B'}{A}$,
$\varphi((a_g)_{g\in A'})\leftrightarrow\varphi({(a_g)}_{g\in A''})$.
By compactness, we have an $L_0$-indiscernible $(a'_g)_{g\in G}$ based on $(a_g)_{g\in G}$, since for given $(a_g)_{g\in G}$ the statement ``$(x_g)_{g\in G}$ is based on $(a_g)_{g\in G}$" can be expressed by a set of $L$-formulas. 
\end{proof}

The following corollary is our main tool in the next section.
\begin{corollary}\rm\label{cor: random hypergraph indiscernibles exist}
Let $G$ be one of the following structures $G_n$, $G_{n,p}$ or $O_{n,p}$,
and let $\bar{a}=\left(a_{g}\right)_{g\in G}$ be given. 
Then there is a $G$-indiscernible $\left(b_{g}\right)_{g\in G}$ based on $\bar{a}$.
\end{corollary}
\begin{proof}
Combining Fact \ref{fac: hypergraphs are Ramsey}, Fact \ref{Ramsey property of ordered partite hypergraphs}, Fact \ref{Ramsey property of ordered partite sets} and Fact \ref{modeling property}.
\end{proof}

We see the most basic application of the above corollary.

\begin{remark}[Existence of an $L_{\op}$-indiscernible witness]\label{indiscernible witness}\rm
In the definition of $\IP_n$, the index set of a witness of $\IP_n$ is $\omega^n$.
By compactness, we can replace $\omega^n$ by any $P_0\times\ldots \times P_{n-1}$ with infinite sets $P_i$ $(i<n)$. Put $G=P_0\sqcup\ldots \sqcup P_{n-1}$ and note that it can be seen as an $L_{\op}$-structure. In this situation, we say that $(a_g)_{g\in G}$ is a witness of $\IP_n$ for $\phi$ if for any two disjoint subsets $X_0$ and $X_1$ of $P_0\times\ldots \times P_{n-1}$ we have that
$$
\{\varphi(x,a_{g_{0}},\ldots ,a_{g_{n-1}})\}_{(g_{0},\ldots ,g_{n-1})\in X_{0}}\cup\{\neg\varphi(x,a_{g_{0}},\ldots ,a_{g_{n-1}})\}_{(g_{0},\ldots ,g_{n-1})\in X_{1}}.
$$
is consistent. Furthermore, observe that if $(b_g)_{g\in O_{n,p}}$ is an $L_{\op}$-indiscernible based on $(a_g)_{g\in O_{n,p}}$, then $(b_g)_{g\in O_{n,p}}$ is also an witness of $\IP_n$ since the $L_{\op}$-isomorphism $X_0X_1\cong_{L_{\op}} Y_0Y_1$ implies that $Y_0$ and $Y_1$ are disjoint subsets of $P_0\times\ldots \times P_{n-1}$ as well.
\end{remark}


\section{$\IP_{n}$ and random hypergraph indiscernibles} \label{sec: IPn and gen indisc}
Recall that $L_{\opg}$ and $L_{\op}$ denote the languages $\left\{<,R,P_0,\ldots ,P_{n-1}\right\}$ 
and $\left\{<,P_0,\ldots ,P_{n-1}\right\}$ respectively.
In this section, we give characterizations of $n$-dependence using $L_{\opg}$-indiscernibles and $L_{\op}$-indiscernibles.
\subsection{Basic properties of $\IP_n$ and indiscernible witnesses}

We begin with some easy remarks on $n$-dependence.
\begin{remark}\rm
\begin{enumerate}
\item
A theory $T$ is $n$-dependent if and only if $T(A)$ is $n$-dependent
for every parameter set $A$. In fact, if $\varphi(x,y_0,\ldots ,y_{n-1}, A)$ has $\IP_n$ in $T(A)$ witnessed by $(a_g)_{g\in O_{n,p}}$, then $\psi(x, z_0,\ldots ,z_{n-1})$ has $\IP_n$ witnessed by $(b_g)_{g\in O_{n,p}}$ where 
$z_i=y_iw$, $\psi(x, z_0,\ldots ,z_{n-1})=\varphi(x,y_0,\ldots ,y_{n-1},w)$ and $b_g=a_gA$.
\item
Let $x\subseteq w$ and $y_i\subseteq z_i$ $(i<n)$ be variables.
If $\varphi(x,y_0,\ldots ,y_{n-1})$ is $n$-dependent then so is $\psi(w,z_0,\ldots ,z_{n-1})=\varphi(x,y_0,\ldots ,y_{n-1})$.
In other words, $n$-dependence is preserved under adding dummy variables.
\item
Suppose that $T$ admits quantifier elimination. If there is no atomic formula having $\IP_n$, then $T$ is $\NIP_n$. This follows from Corollary \ref{boolean combinations}.
\end{enumerate}
\end{remark}

For an $n$-partite $n$-uniform hypergraph $(G,R,P_0,\ldots ,P_{n-1})$, we say
a formula $\varphi(x_0,\ldots ,x_{n-1})$ encodes $G$ if there is a $G$-indexed set 
$(a_g)_{g\in G}$ such that $\models \varphi(a_{g_0},\ldots ,a_{g_{n-1}}) \Leftrightarrow R(g_0,\ldots ,g_{n-1})$ for every $g_i\in P_i$.

\begin{proposition}\label{encodes partite random graph}\rm
Let $\varphi(x,y_0,\ldots ,y_{n-1})$ be a formula.
The following are equivalent.
\begin{enumerate}
\item
$\varphi$ has $\IP_n$.
\item
$\varphi$ encodes every $(n+1)$-partite $(n+1)$-uniform hypergraph $G$.
\item
$\varphi$ encodes $G_{n+1,p}$ as a partite hypergraph.
\item
$\varphi$ encodes $G_{n+1,p}$ as a partite hypergraph by a $G_{n+1,p}$-indiscernible $(a_g)_{g\in G_{n+1,p}}$.
\end{enumerate}
\end{proposition}
\begin{proof}
(1)$\Rightarrow$ (2):
By compactness, it is enough to check it for every finite hypergraph $G$ with
$|P_0(G)|=\ldots =|P_n(G)|=k$.
Let $(a_g)_{g\in O_{n,p}}$ be a witness of $\IP_n$ of $\varphi$.
Let $V_i\subset P_i(O_{n,p})$ be a $k$-point subset for each $i<n$.
For simplicity, we consider $V_i = P_{i+1}(G)$.
For $g\in P_0(G)$, let $X_g=\{(g_0,\ldots ,g_{n-1}): G\models R(g,g_0,\ldots ,g_{n-1}), g_i\in V_i\}$.
By the definition of $\IP_n$, we can find $b_g$ such that
$$
\varphi\left(b_{g},a_{g_0},\ldots ,a_{g_{n-1}}\right)\Leftrightarrow\left(g_0,\ldots ,g_{n-1}\right)\in X_g\mbox{.}
$$
Then letting $a_g= b_g$ for $g\in P_0(G)$, we have that $(a_g)_{g\in G}$ witnesses that $\varphi$ encodes $G$.

(2)$\Rightarrow$(3):
Trivial.

(3)$\Rightarrow$(4):
Suppose that $(a_g)_{g\in G_{n+1,p}}$ witnesses (3).
By Corollary \ref{cor: random hypergraph indiscernibles exist}, there is  a $G_{n+1,p}$-indiscernible $(b_g)_{g\in G_{n+1,p}}$ based on $(a_g)_{g\in G_{n+1,p}}$, which then also witnesses that $\varphi$ encodes $G_{n+1,p}$ as a partite hypergraph. 

(4)$\Rightarrow$(1):
Since $G_{n+1,p}$ is random, the set $\{a_g: g\in P_i(G_{n+1,p}), i>0\}$ witnesses $\IP_n$ for $\varphi$.
\end{proof}

As any permutation of parts of a countable partite random hypergraph is an automorphism,
we have that $n$-dependence is preserved under rearranging the order of the variables (in particular, one can exchange the roles of the free variable and a parameter variable):  

\begin{corollary}\rm \label{cor: permuting vars}
Let $\varphi(x,y_0,\ldots ,y_{n-1})$ be a formula.
Suppose that $(w,z_0,\ldots ,z_{n-1})$ is any permutation of the sequence $(x,y_0,\ldots ,y_{n-1})$.
Then $\psi(w,z_0,\ldots ,z_{n-1})$ is $n$-dependent if and only if $\varphi(x,y_0,\ldots ,y_{n-1})$ is $n$-dependent, where $\psi(w,z_0,\ldots ,z_{n-1})=\varphi(x,y_0,\ldots ,y_{n-1})$.
\end{corollary}

\subsection{Characterizations of $\NIP_n$ by collapsing indiscernibles}

Recall that $(G_n, <, R)$ is a countable ordered $n$-uniform random hypergraph, and $(G_{n,p},<,R,P_0,\ldots ,P_{n-1})$ is a countable ordered $n$-partite $n$-uniform random hypergraph. In this subsection we prove the following theorem. 

\begin{theorem}\label{thm: generalized Skow}\rm
The following are equivalent:
\begin{enumerate}
\item
$T$ is $n$-dependent,
\item
every $G_{n+1,p}$-indiscernible is actually $L_{\op}$-indiscernible,
\item
every $G_{n+1}$-indiscernible
is actually order indiscernible, i.e. $\{<\}$-indiscernible.
\end{enumerate}
\end{theorem}

\begin{remark}\rm
\begin{enumerate}
\item
When $n=1$, this 
is due to Scow \cite{Scow:2011uq}.
\item
In the theorem, (2)$\Rightarrow$(1) follows immediately from Proposition \ref{encodes partite random graph}, since if $\varphi$ encodes $G_{n+1,p}$ by $(a_g)_{g\in G_{n+1,p}}$,
then $(a_g)_{g\in G_{n+1,p}}$ cannot be an $L_{\op}$-indiscernible.
\item
This characterization suggests that strongly minimal theories may be considered as ``$0$-dependent'' theories.
\end{enumerate}
\end{remark}
First we discuss (3)$\Rightarrow$(2).
Let $P_i$ be the $i$-th part of $G_{n+1,p}$.
Since $P_i$ is order isomorphic to $\Q$, we may assume $P_i=\{g_q^i: q\in \Q\}$ with $g_q<g_p \Leftrightarrow q<p$.
Let $G^*_{n+1}$ be an ordered $(n+1)$-uniform hypergraph defined as
\begin{itemize}
\item
$G^*_{n+1}=\{h_q: h_q = (g^0_q, \ldots ,g^{n}_q), q\in \Q\}$,
\item
$\{h_{q_0}, \ldots ,h_{q_n}\}\in R(G^*_{n+1})\Leftrightarrow R(g^0_{q_0}, \ldots ,g^n_{q_n})$ for $q_0< \ldots <q_n$,
\item
$h_q<h_p\Leftrightarrow q<p$
\end{itemize}
for every $q,p\in \Q$.
The hypergraph $G^*_{n+1}$ is clearly $K$-universal where $K$ is the class of all finite ordered $n$-uniform hypergraphs.
\begin{proof}[Proof of (3)$\Rightarrow$(2) of Theorem \ref{thm: generalized Skow}]
Let $(a_g)_{g\in G_{n+1,p}}$ be a $G_{n+1,p}$-indiscernible which is not $L_{\op}$-indiscernible. 
We construct a $G_{n+1}$-indiscernible which is not order indiscernible.
Assume that $G_{n+1}=\{g_q^i: i<n+1, q\in \Q\}$ as discussed above.
By the assumption there are $A \cong_{L_{\op}} B\subset G_{n+1,p}$ such that
$\tp((a_g)_{g\in A})\neq \tp((a_g)_{g\in B})$.
Without loss of generality, we may assume that if $g^i_q, g^j_p\in A$ and $i<j$ then $q<p$, and the same for $B$.
For $h_q\in G^*_{n+1}$, let $b_{h_q}=(a_{g^0_q}, \ldots ,a_{g^n_q})$ and consider $G^*_{n+1}$-indexed set $(b_{h})_{h\in G^*_{n+1}}$.
Let $A^*=\{h_q: g^i_q\in A\}$ and $B^*=\{h_q: g^i_q\in B\}$. Fix the tuple $y_q = (x_q^0, \ldots,y_q^{n-1})$ of variables for each element $h_q = (g^0_q,\ldots,g^{n-1}_q)\in A^*$,  and let $\varphi^*((y_q)_{h_q\in A^*}) = \varphi((x^i_q)_{g^i_q\in A})$ for each formula $\varphi((x^i_q)_{g^i_q\in A})$. Then for any formula $\varphi((x^i_q)_{g^i_q\in A})$, we have $\varphi^*((b_q)_{h_q\in A^*})\leftrightarrow \varphi^*((b_q)_{h_q\in X})$ 
whenever $A^*\cong_{<,R} X\subset G^*_{n+1}$ (and the same holds for $B^*$).\footnote{We thank Adri\'an Portillo for pointing out an inaccuracy in the previous version of the proof.}  
Applying Fact \ref{modeling property} to $G^*_{n+1}$, we have a $G^*_{n+1}$-indiscernible $(b'_h)_{h\in G^*_{n+1}}$ based on $(b_h)_{h\in G^*_{n+1}}$.
By the construction, $(b'_h)_{h\in G^*_{n+1}}$ is not order indiscernible indiscernible, because $\varphi^*((b_h)_{h\in A^*})\not\leftrightarrow \varphi^*((b_h)_{h\in B^*})$ for some $\varphi((x^i_q)_{g^i_q\in A})$ while $A^* \cong_< B^*$.
Finally, by compactness, we can find $(c_g)_{g\in G_{n+1}}$ that is $G_{n+1}$-indiscernible but not order indiscernible.
\end{proof}

Now we work towards the converse.
Although the remaining part is only (1)$\Rightarrow$(3), we see both (1)$\Rightarrow$(3) and (1)$\Rightarrow$(2) with the same method because a proposition proved in the second one is used in the next section. So let $\left( G_*, L_{o*}, L_{g*} \right)$ be either $\left(G_n, \{<\}, \{<,R\}\right)$ or $\left(G_{n,p}, L_{\op}, L_{\opg}\right)$.

Let $V\subset G_{*}$ be a finite set and $g_0,\ldots ,g_{n-1}, g'_0,\ldots ,g'_{n-1}\in G_*\setminus V$ such that $R(g_0,\ldots ,g_{n-1})\not\leftrightarrow R(g'_0,\ldots ,g'_{n-1})$.
Then $W=g_0\ldots g_{n-1}V$ is said to be \emph{$V$-adjacent} to $W'=g'_0\ldots g'_{n-1}V$ if
\begin{itemize}
\item
$W\cong_{L_{o*}} W'$,
\item
for every nonempty $\bar v\in V$  with $|\bar v|=k$ and $i_0,\ldots,i_{n-k-1}<n$
$$
R(g_{i_0},\ldots ,g_{i_{n-k-1}},\bar v)\leftrightarrow R(g'_{i_0},\ldots ,g'_{i_{n-k-1}},\bar v).
$$
\end{itemize}
Recall that $R$ is a symmetric relation, so we do not care about order permutations of substituted elements.
$W$ is said to be \emph{adjacent} to $W'$ if there is $V\subset W\cap W'$ such that $W$ is $V$-adjacent to $W'$.
Roughly speaking, $W$ is adjacent to $W'$ if $W$ can be made isomorphic to $W'$ by adding or deleting an edge.  
\begin{lemma}\label{nORG}\rm
Let $W, W'\subset G_{*}$ be subsets such that $W\cong_{L_{o*}} W'$.
Then there is a sequence $W=W_0, W_1,\ldots, W_k$ such that $W_{i+1}$ is adjacent to $W_i$ for every $i<k$ and $W_k\cong_{L_{g*}}W'$.
\end{lemma}
\begin{proof}
The proof is the same for both $\left(G_n, \{<\}, \{<,R\}\right)$ and $\left(G_{n,p}, L_{\op}, L_{\opg}\right)$. We only prove the statement for $\left( G_{p,n}, L_{\op}, L_{\opg}\right)$.
If $W\cong_{L_{\opg}} W'$ then there is nothing to show, so we may assume that $W\not\cong_{L_{\opg}} W'$.
Consider any $g_i\in W$ ($i<n$) such that $g_i \in P_i$, and let $V=W\setminus\{g_0,\ldots,g_{n-1}\}$.
By Remark \ref{Theory of random graph}, we can find $g'_0\in G_{n,p}$
such that $g_0g_1\ldots g_{n-1}V$ is $V$-adjacent to $g'_0g_1\ldots g_{n-1}V$.
This means that we can change the existence of any edge by moving a vertex,
and get the required sequence.
\end{proof}

\begin{lemma}\label{random subgraph}\rm
Let $V\subset G_{*}$ be a finite set and let $g_0<\ldots <g_{n-1} \in G_{*} \setminus V$ with $R(g_0,\ldots ,g_{n-1})$.
Then there are infinite sets $X_0<\ldots <X_{n-1}\subseteq G_{*}$ such that
\begin{itemize}
\item
$(G'; <, R) \cong (G_{n,p}; <, R)$ where $G'=X_0\ldots X_{n-1}$,
\item
for any $g'_i\in X_i$ $(i<n)$, either $W\cong_{L_{\opg}} W'$ or $W$ is $V$-adjacent to $W'$, where 
$W=g_0\ldots g_{n-1}V$ and $W'=g'_0\ldots g'_{n-1}V$.
\end{itemize}
\end{lemma}
\begin{proof}
Again, the same argument works for both cases; we deal with partite graphs.
Let $g_0,g_1,\ldots $ be an enumeration of $G_{p,n}$.
We choose $G'=\{h_i\}_{i\in\omega}$ by recursion on $i$. First set $h_i=g_i$ for $i<n$. Suppose now that we have already obtained $h_0,\ldots, h_{m-1}$ for some $m\geq n$.
Since $G_{n,p}$ is random, we can find $h_m\in G_{n,p}$ such that $h_mV\cong_{L_{\op}} g_iV$ with $i$ taken in such a way that $g_m\in P_i$, that $g_0\ldots g_m \cong_{L_{\opg}} h_0\ldots h_m$, and 
for every nonempty $\bar v\in V$ 
$$
R(g_m, g_{i_0},\ldots ,g_{i_{k-1}},\bar v)\leftrightarrow R(h_m, h_{i_0},\ldots ,h_{i_{k-1}},\bar v).
$$
Finally, note that $X_i=P_i(G')$ satisfies the requirements. 
\end{proof}

Now we prove that the existence of a $G_{*}$-indiscernible which is not $L_{o*}$-indiscernible implies $\IP_n$.
We carefully discuss how to find a witness of $\IP_n$.
\begin{proposition}\label{prop:reduction to order indisc implies NIP}\rm
Suppose that there
is a $G_{n+1,p}$-indiscernible $(a_g)_{g\in G_{n+1,p}}$ that is not $L_{\op}$-indiscernible. 
Then there are a finite set $V\subset G_{n+1,p}$, an $L(A)$-formula $\varphi(x,y_0,\ldots ,y_{n-1},A)$ with $A=(a_g)_{g\in V}$ and a subgraph $G'\subset G_{n+1,p}$ with $G'\cong_{L_{\opg}}G_{n+1,p}$ such that 
$\varphi(x,y_0,\ldots ,y_{n-1})$ encodes $G'$ by $(a_g)_{g\in G'}$.
In particular, the formula $\varphi(x,y_0,\ldots ,y_{n-1},A)$ has $\IP_n$.
Moreover, for every $W,W'\subset G'$, we have $WV\cong_{L_{\op}} W'V$ whenever $W\cong_{L_{\op}} W'$.
\end{proposition}
\begin{proof}
Since $(a_g)_{g\in G_{n+1,p}}$ is not $L_{\op}$-indiscernible, there are some subsets $W,W'\subset G$ with $W\cong_{L_{\op}} W'$
and an $L$-formula $\varphi(x_{0},\ldots ,x_{m-1})$ such that $\varphi((a_g)_{g\in W})\wedge\neg\varphi((a_{g'})_{g'\in W'})$ holds.
Without loss, we may assume by Lemma \ref{nORG} that
$W$ is $V$-adjacent to $W'$ for some subset $V$ such that $W=g_0g_1\ldots g_{n}V$, $W'=g'_0g'_1\ldots g'_{n}V$, and $R(g_0,\ldots ,g_n)\wedge\neg R(g'_0,\ldots ,g'_n)$.
Now, let  $G'\subset G_{n+1,p}$ be a subgraph obtained after applying Lemma \ref{random subgraph} to $V$ and $g_0\ldots g_{n}$.
Then for every $h_i\in P_i(G')$ $(i<n+1)$, 
we have $$R(h_0,\ldots ,h_{n}) \Leftrightarrow h_0\ldots h_{n}V\cong_{L_{\opg}} W$$
and
$$\neg R(h_0,\ldots ,h_{n}) \Leftrightarrow h_0\ldots h_{n}V\cong_{L_{\opg}} W'.$$
Since $(a_g)_{g\in G_{n+1,p}}$ is $G_{n+1,p}$-indiscernible, we have that
$\varphi(a_{h_0},\ldots ,a_{h_{n}}, A)$ holds if and only if $R(h_0,\ldots ,h_{n})$ holds, where $A=(a_g)_{g\in V}$.
Thus, the fact that the relation $R$ is random on $G'$ implies that $\varphi(x, y_0,\ldots ,y_{n-1},A)$ has $\IP_n$.
Finally, the moreover part is immediate from the definition of $G'$. 
\end{proof}

\begin{proposition}\rm
Suppose that there is a $G_{n+1}$-indiscernible which is not $<$-indiscernible.
Then $T$ has $\IP_n$.
\end{proposition}
\begin{proof}
The proof is the same as for Proposition \ref{prop:reduction to order indisc implies NIP}.
\end{proof}

\section{Reduction to 1 variable} \label{sec: reduction to 1 variable}
Recall that $(G_{n,p};<,R,P_0,\ldots ,P_{n-1})$ is a countable ordered $n$-partite $n$-uniform random hypergraph.

A standard characterization of dependence of a formula in terms of finite alternation on an infinite indiscernible sequence (see e.g. \cite[Proposition 4]{adler2008introduction}) can be easily reformulated using Ramsey and compactness in the following way:

\begin{remark}\rm
Let $R$ be a dense co-dense subset of $\mathbb{Q}$. The following
are equivalent:
\begin{enumerate}
\item $\varphi(x,y)$ has $\IP$.
\item There are $b$ and $(a_{i})_{i\in\mathbb{Q}}$ such that:
\begin{enumerate}
\item $(a_{i})_{i\in\mathbb{Q}}$ is order indiscernible,
\item $\models\varphi(b,a_{i})$ if and only if $i\in R$, for all $i\in\mathbb{Q}$,
\item in addition, $(a_{i})_{i\in\mathbb{Q}}$ is $(<,R)$-indiscernible
over $b$.
\end{enumerate}
\end{enumerate}
\end{remark}

We give an appropriate generalization for $n$-dependence (recall that an ordered random $1$-hypergraph is just a dense linear order with a dense co-dense subset).

\begin{lemma}\label{lem: A characterization of IP_n}\rm
The following are equivalent. 
\begin{enumerate}
\item $\varphi(x,y_{0},\ldots ,y_{n-1})$ has $IP_{n}$.
\item There are $b$ and $(a_{g})_{g\in G_{n}}$ such that:
\begin{enumerate}
\item $(a_{g})_{g\in G_{n,p}}$ is $L_{\op}$-indiscernible,
\item $\models\varphi(b,a_{g_{0}},\ldots ,a_{g_{n-1}})$ if and only if $R(g_{0},\ldots ,g_{n-1})$,
for all $g_{i}\in P_{i}$. 
\end{enumerate}
\item
There are $b$ and $(a_g)_{g\in G_{n,p}}$ satisfying the conditions (2a), (2b) and
\begin{itemize}
\item[(c)] $(a_{g})_{g\in G_{n,p}}$ is $G_{n,p}$-indiscernible over $b$.
\end{itemize}
\end{enumerate}
\end{lemma}
\begin{proof}
(3) $\Rightarrow$ (2): Trivial.

(2) $\Rightarrow$ (1): We show that $(a_{g})_{g\in G_{n,p}}$
witnesses the $n$-independence of $\varphi$. Let $X_{0}$ and $X_{1}$
be any disjoint finite subsets of $P_{0}\times\ldots \times P_{n-1}$.
As $G_{n,p}$ is random, there are subsets $X_{0}'X_{1}'\cong_{L_{\op}}X_{0}X_{1}$
of $P_{0}\times\ldots \times P_{n-1}$ such that 
$X_{0}'\subset R$ and
$X_{1}'\subset R^{c}=\Pi_i P_i \setminus R$.
By the assumption, observe that $\varphi(b,a_{g_{0}},\ldots ,a_{g_{n-1}})$ holds only when $(g_{0},\ldots ,g_{n-1})\in X_{0}'$, and hence
\[
\models\exists x\bigwedge_{i<2}\bigwedge_{(g_{0},\ldots ,g_{n-1})\in X_{i}}\varphi(x,a_{g_{0}},\ldots ,a_{g_{n-1}})^{{\rm {if}(i=0)}}
\]
since $(a_{g})_{g\in G_{n}}$ is $L_{\op}$-indiscernible.

(1) $\Rightarrow$ (3):
Suppose that $\varphi(x,y_0,\ldots ,y_{n-1})$ has $\IP_n$.
By Remark \ref{indiscernible witness}, we may assume the existence of an $L_{\op}$-indiscernible $(a'_g)_{g\in G_{n,p}}$ witnessing that $\varphi$ has IP$_n$.
Then there is some $b$ such that $\models\varphi(b,a'_{g_{0}},\ldots ,a'_{g_{n-1}})$ if and only if $R(g_{0},\ldots ,g_{n-1})$, for all $g_{i}\in P_{i}$. 
Let $(a_g)_{g\in G_{n,p}}$ be a $G_{n,p}$-indiscernible based on $(a'_g)_{g\in G_{n,p}}$ over $b$. Clearly, any such sequence satisfies conditions (b) and (c). 
To see that $(a_g)_{g\in G_{n,p}}$ is an $L_{\op}$-indiscernible sequence, consider some finite subsets $W$ and $V$ of $G_{n,p}$ with $W\cong_{L_{\op}} V$, and a formula $\theta((x_g)_{g\in W})$. 
We show that $\theta((a_g)_{g\in W})$ holds if and only if $\theta((a_g)_{g\in V})$ holds.
Since $(a_g)_{g\in G_{n,p}}$ is based on $(a'_g)_{g\in G_{n,p}}$, there is $W'V'\cong_{L_{\opg}} WV$ such that
$$\theta((a_g)_{g\in W})\leftrightarrow \theta((a'_g)_{g\in W'})\ \mbox{ and } \ \theta((a_g)_{g\in V})\leftrightarrow \theta((a'_g)_{g\in V'})$$ hold.
Now, the fact that $W'\cong_{L_{\op}}V'$ yields that $\theta((a'_g)_{g\in W'})$ holds if and only if $\theta((a'_g)_{g\in V'})$ holds, as desired.
\end{proof}

\begin{proposition}\label{lem: char of NIP_n by preserving indisc}\rm
The following are
equivalent:
\begin{enumerate}
\item Every $L$-formula $\phi\left(x,y_{0},\ldots,y_{n-1}\right)$ with $\left|x\right| \leq m$
is $n$-dependent.
\item For any $\left(a_{g}\right)_{g\in G_{n,p}}$ and $b$ with $\left|b\right| = m$,
if $\left(a_{g}\right)_{g\in G_{n,p}}$ is $G_{n,p}$-indiscernible over
$b$ and $L_{\op}$-indiscernible (over $\emptyset$), then it is $L_{\op}$-indiscernible over $b$.
\end{enumerate}
\end{proposition}
\begin{proof}
(1) $\Rightarrow$ (2): Let $\left(a_{g}\right)_{g\in G_{n,p}}$ be $G_{n,p}$-indiscernible over $b$ with $\left|b\right|= m$, and set $a_{g}'=ba_{g}$. Thus $\left(a_{g}'\right)_{g\in G_{n,p}}$ is $G_{n,p}$-indiscernible (over $\emptyset$). Suppose, towards a contradiction, that $\left(a_{g}\right)_{g\in G_{n,p}}$ is not $L_{\op}$-indiscernible over $b$; in other words
$\left(a_{g}'\right)_{g\in G_{n,p}}$ is not $L_{\op}$-indiscernible (over $\emptyset$).
By Proposition \ref{prop:reduction to order indisc implies NIP},
there is a subgraph $G'\subset G_{n,p}$, a finite set $V\subset G_{n,p}$, and
a formula $\psi(y'_0,\ldots ,y'_{n-1},\bar z')$ 
such that 
\begin{enumerate}
\item
$G'\cong_{L_{\opg}} G_{n,p}$,
\item
$R(g_0,\ldots ,g_{n-1})$ holds if and only if $\psi(a'_{g_0},\ldots ,a'_{g_{n-1}},(a'_g)_{g\in V})$ for every $g_i\in P_i(G')$,
\item
for every $W, W'\subset G'$ we have $WV\cong_{L_{\op}}W'V$ whenever  $W\cong_{L_{\op}}W'.$
\end{enumerate}
Now, observe that each variable $y_i'$ is of the form $xy_i$ where $x$ corresponds to the tuple $b$, and the variable $\bar z'$ is of the form $x\bar z$. Let $\varphi(x,z_0,\ldots z_{n-1})$ be the formula $\psi(y'_0,\ldots ,y'_{n-1},\bar z')$ where $z_i=y_i\bar z$. Moreover, for each $g\in G'$, let $c_g=a_g(a_{g'})_{g' \in V}$ and note that $\varphi(b,c_{g_0},\ldots ,c_{g_{n-1}})$ holds if and only if $R(g_0,\ldots ,g_{n-1})$ does. As the sequence $(a_g)_{g\in G_{n,p}}$ is $L_{\op}$-indiscernible, so is $(c_g)_{g\in G'}$, and therefore the formula $\varphi(x,y_0,\ldots y_{n-1})$ has $\IP_n$ by Lemma \ref{lem: A characterization of IP_n}.

(2) $\Rightarrow$ (1): 
Immediate from Lemma \ref{lem: A characterization of IP_n}.
\end{proof}

The following theorem is from \cite[Section 2]{Shelah:vn}. We are following the same strategy as the proof there, however the authors felt that a more detailed account could be provided.
\begin{theorem}[Shelah] \label{thm: reduction to one var}\rm
$T$ is $n$-dependent if and only if every $L$-formula
$\varphi(x,y_{0},\ldots ,y_{n-1})$ with $|x|=1$ is $n$-dependent.
\end{theorem}
\begin{proof}
We show by induction that if the condition (2) of Proposition \ref{lem: char of NIP_n by preserving indisc}
holds for $m=1$, then it holds for all $m\in\omega$. 
Let $\bar{b}=b_{0}\ldots b_{m}$ be a tuple with $|b_i|=1$, and let $\left(a_{g}\right)_{g\in G_{n,p}}$ be given such that $\left(a_{g}\right)_{g\in G_{n,p}}$ is $G_{n,p}$-indiscernible over $\bar{b}$.
Note that $\left(a_{g}\right)_{g\in G_{n,p}}$ is $L_{\op}$-indiscernible over $b_{m}$, as otherwise 
it is not $L_{\op}$-indiscernible over $b_m$ but $G_{n,p}$-indiscernible over $b_m$, contradicting the inductive assumption. 
Now, consider the sequence $\left(b_{m}a_{g}\right)_{g\in G_{n,p}}$ and notice that it is clearly $G_{n,p}$-indiscernible over $b_{0}\ldots b_{m-1}$. Applying
the inductive assumption again, we conclude that $\left(b_{m}a_{g}\right)_{g\in G_{n,p}}$
is $L_{\op}$-indiscernible over $b_{0}\ldots b_{m-1}$,
which implies that $\left(a_{g}\right)_{g\in G_{n}}$ is 
$L_{\op}$-indiscernible over $b_{0}\ldots b_{m}$, as desired.
\end{proof}

Finally, we summarize the basic properties of $n$-dependent theories established throughout the paper, giving a criterion for $n$-dependence of a theory.

\begin{proposition} \label{prop: criterion for n-dependence}\rm
\begin{enumerate}
\item Boolean combinations preserve $n$-dependence (Corollary \ref{boolean combinations}).
\item Permuting variables preserves $n$-dependence (Corollary \ref{cor: permuting vars}).
\item Failure of $n$-dependence of a theory is witnessed by a formula in a single free variable (Theorem \ref{thm: reduction to one var}).
\item If $T$ eliminates quantifiers, then in order to check that $T$ is $n$-dependent it is enough to check that every atomic formula in a single free variable is $n$-dependent, e.g. by checking that the number of $\phi$-types is not maximal (combining (1) and (3) above). 
\end{enumerate}
\end{proposition}

\appendix

\section{Ramsey property for hypergraphs}
In this section we verify that the two classes of structures considered in the previous sections have Ramsey property.
First we see that the class of finite ordered $n$-partite sets is Ramsey, and then that the class of all finite (linearly) ordered $n$-partite $n$-uniform hypergraphs is also Ramsey.
Basic notation and definitions are already given in Section 4.1, so we don't repeat them.

For a given class $K$ of $L$-structures, let $K^*$ be the hereditary closure of $K$, i.e. $K^*=\{A: A\subseteq B, B\in K\}$.
\begin{lemma}\label{hereditary closure}\rm
Let $K$ be a set of $L$-structures satisfying Ramsey property.
Suppose that every $A\in K$ has no non-trivial automorphisms, i.e. $Aut(A)=\{id_A\}$.
If the hereditary closure $K^*$ of $K$ has the amalgamation property, then $K^*$ has Ramsey property.
\end{lemma}
\begin{proof}
Let $A, B\in K^*$. 
Fix an extension $A\subseteq A_0\in K$ and consider a structure $A_0'\cong A_0$.
Note that, in general, $\binom{A_0'}{A}$ is not a singleton. 
However, by the assumption, we can recognize a unique $A'\subset A_0'$ which corresponds to $A\subset A_0$ through the unique isomorphism $A_0'\cong A_0$.
By applying amalgamation property, we have an extension $B_0\in K$ of $B$ such that for every $A'\in \binom{B}{A}$ there is an extension $A'\subseteq A_0' \in\binom{B_0}{A_0}$ which is isomorphic to the extension $A\subseteq A_0$. 
Since $K$ has Ramsey property, we can find $C\in K$ such that $C\to(B_0)^{A_0}_k$.
It is easy to check that $C\to(B)^A_k$, since any coloring $c:\binom{C}{A}\to k$ induces a coloring $\tilde c(A_0')=c(A')$ for $A'\subset A_0'\in \binom{C}{A_0}$, hence
there is $B_0'\subset C$ on which $\tilde c$ is constant. Clearly, for a $B'\subset B_0$, $c$ is constant on $\binom{B'}{A}$.
\end{proof}

For $L$-structures $A$ and $B$, let $A\oplus B$ be an 
 $L\cup\{P_0(x), P_1(x)\}$-structure such that $P_0=A$, $P_1=B$ and $A\oplus B=P_0\sqcup P_1$.
Let $K_0$ and $K_1$ be two classes of $L$-structures.
We define a class $K_0\oplus K_1$ of $L\cup\{P_0(x), P_1(x)\}$-structures by $K_0\oplus K_1=\{A_0\oplus A_1: A_0\in K_0, A_1\in K_1\}$.
\begin{lemma}\label{direct sum}\rm
If $K_0$ and $K_1$ have Ramsey property, then so does $K_0\oplus K_1$.
\end{lemma}
\begin{proof}
Let $A_0\oplus A_1, B_0\oplus B_1\in K_0\oplus K_1$.
Fix $C_1\in K_1$ such that $C_1\to(B_1)^{A_1}_k$.
Let $m = \left|\binom{C_1}{A_1}\right|$.
We can find $C_0^0,\ldots ,C_0^m=C_0\in K$ such that $C_0^0\to (B_0)^{A_0}_k$ and $C_0^{i+1}\to(C_0^i)^{A_0}_k$ for $i\leq m$.
We show that $C_0\oplus C_1\to(B_0\oplus B_1)^{A_0\oplus A_1}_k$.
Let $c:\binom{C_0\oplus C_1}{A_0\oplus A_1}\to k$ be a coloring.
Then for each $A_1' \in \binom{C_1}{A_1}$, we have an induced coloring
$c_{A_1'}: \binom{C_0}{A_0}\to k$ such that $c_{A_1'}(A_0')=c(A_0'\oplus A_1')$. 
By the construction, we can find $B_0'\in \binom{C_0}{B_0}$ such that $c_{A_1'}$ is constant on $\binom{B_0'}{A_0}$ for every ${A_1'}\in\binom{C_1}{A_1}$.
Then, the values of $c_{A_1'}$ on $\binom{B_0'}{A_0}$ define a coloring 
$\tilde c: \binom{C_1}{A_1}\to k$ by $\tilde c(A_1')=c_{A_1'}(A_0')$ where $A_0'\subset B_0'$.
Hence there is $B_1'\in\binom{C_1}{B_1}$ such that $\tilde c$ is constant on $\binom{B_1'}{A_1}$. 
Therefore, $c$ is constant on $\binom{B_0'\oplus B_1'}{A_0\oplus A_1}$. 
\end{proof}
The classical Ramsey theorem implies that the class of all finite linearly ordered sets has Ramsey property. Therefore, with the above lemmas, we have the following:

\begin{proposition}\label{partite ordered sets has Ramsey property}\rm
Let $K$ be the set of finite ordered $n$-partite sets.
Then both $K$ and its hereditary closure $K^*$ have Ramsey property.
\end{proposition}

Next we'll prove that the set of finite ordered $n$-partite $n$-uniform hypergraphs has Ramsey property. 
Let $R$ be an $n$-place relation for some $n\geq 1$.
Recall that an ordered $n$-uniform hypergraph is an $L$-structure $A$ such that $R$ is symmetric and irreflexive on $A$ and that $<$ is a linear ordering on $A$.
Our starting point is the following well-known fact:
\begin{fact}[Nes\'etril, R\"odl \cite{MR0437351,MR692827}; Abramson, Harrington \cite{MR503795}]\rm\label{graph Ramsey}
Let $K$ be the set of all finite ordered $n$-uniform hypergraphs.
Then $K$ has Ramsey property.
\end{fact}

Recall that $L_{\opg}=\{R(x_0,\ldots ,x_{n-1}),<, P_0(x),\ldots , P_{n-1}(x)\}$ and that
an ordered $n$-partite $n$-uniform hypergraph is an $L_{\opg}$-structure $A$ satisfying the following:
\begin{enumerate}
\item
$A|\{R,P_0,\ldots ,P_{n-1}\}$ is an $n$-partite $n$-uniform hypergraph,
\item
$<$ is a total ordering on $A$ satisfying $P_0(A)<\ldots <P_{n-1}(A)$.
\end{enumerate}

\begin{proposition}\label{ramsey property of partite graphs}\rm
Let $K$ be the set of finite ordered $n$-partite $n$-uniform hypergraphs.
Then $K$ has Ramsey property.
\end{proposition}
\begin{proof}
Since the general case is similar, we assume $n=2$ for simplicity.
Fix $A, B\in K$ and $k\in \omega$.
Let $A_0=A|\{R,<\}$ and $B_0=B|\{R,<\}$ respectively.
Then there is an ordered graph $C_0$ such that $C_0\to(B_0)^{A_0}_k$.
For a given ordered graph $X_0=\{v_0<\ldots <v_{m-1}\}$,
let $\tilde X_0$ be an ordered bipartite graph such that
\begin{itemize}
\item
$P_i(\tilde X_0)=\{w^i_0<\ldots <w^i_{m-1}\}$,
\item
$R(\tilde X_0)\ni (w^0_i,w^1_j)$ if and only if $i<j$ and $R(v_i, v_j)$ in $X_0$.
\end{itemize}
\begin{claim*}
$\tilde C_0\to(B)^A_k$.
\end{claim*}
Suppose that $X$ is a bipartite graph with $P_0(X)=\{v_0<\ldots <v_{l-1}\}$ and $P_1(X)=\{v_l<\ldots <v_{m-1}\}$.
For the ordered graph $X_0=X|\{R,<\}$, put $\bar X_0=\{w^0_0<\ldots <w^0_{l-1}\}\cup \{w^1_l<\ldots <w^1_{m-1}\}\subset\tilde X_0$.
(So $\bar X_0$ is a bipartite subgraph of $\tilde X_0$.)
One can easily check that $\bar X_0 \cong X$.
With this fact in mind, let $c:\binom{\tilde C_0}{A}\to k$ be any coloring.
Then there is an induced coloring $\tilde c:\binom{C_0}{A_0}\to k$ such that
$\tilde c(A_0')= c(\bar A_0')$ for all $A_0'\in\binom{C_0}{A_0}$.
Let $B_0'\in\binom{C_0}{A_0}$ be such that $\tilde c$ is constant on $\binom{B_0'}{A_0}$.
Then $\bar B_0'\in \binom{\tilde C_0}{B}$ satisfies the required condition. 
\end{proof}

\bibliographystyle{alpha}
\bibliography{IP_n}

\end{document}